\documentclass[11pt]{amsart}

\usepackage{amssymb}
\usepackage{mathptmx,a4}
\usepackage[pdftex]{graphics}
\usepackage{graphicx}
\begin{document}

\newcommand{\mnote}[1]{\marginpar{\tiny\tt #1}}

\newcommand{\CC}{{\mathbb C}}
\newcommand{\RR}{{\mathbb R}}
\newcommand{\ZZ}{{\mathbb Z}}
\newcommand{\NN}{{\mathbb N}}
\newcommand{\K}{{\mathcal K}}

\renewcommand{\phi}{\varphi}
\newcommand{\halb}{\frac{1}{2}}
\renewcommand{\Re}{\operatorname{Re}}
\newcommand{\spec}{\operatorname{spec}}
\newcommand{\scal}{{\rm scal}}
\newcommand{\ric}{{\rm ric}}
\newcommand{\Ric}{{\rm Ric}}
\newcommand{\dV}{{\rm dV}}
\newcommand{\dA}{{\rm dA}}
\newcommand{\area}{{\rm area}}
\newcommand{\supp}{{\rm supp}}
\newcommand{\ind}{\operatorname{ind}}
\newcommand{\nE}{\nabla^E}
\newcommand{\<}{\langle}
\renewcommand{\>}{\rangle}
\newcommand{\spp}{\sigma_{\rm p}}
\newcommand{\scc}{\sigma_{\rm c}}
\newcommand{\sd}{\sigma_{\rm d}}
\newcommand{\se}{\sigma_{\rm e}}
\newcommand{\grad}{{\rm grad}}

\newcommand{\gl}[1]{(\ref{#1})}

\theoremstyle{plain}
\newtheorem{thm}{Theorem}[section]
\newtheorem{lemma}[thm]{Lemma}
\newtheorem{prop}[thm]{Proposition}
\newtheorem{cor}[thm]{Corollary}
\newtheorem{conjecture}[thm]{Conjecture}

\theoremstyle{definition}
\newtheorem{remark}[thm]{Remark}
\newtheorem{definition}[thm]{Definition}
\newtheorem{example}[thm]{Example}

\title{Spectral Bounds for Dirac Operators on Open Manifolds}
\author{Christian B\"ar}
\thanks{}
\date{\today}
\address{Universit\"at Potsdam\\
Institut f\"ur Mathematik, Am Neuen Palais 10, Haus 8, Germany}
\email{baer@math.uni-potsdam.de}
\begin{abstract}
We extend several classical eigenvalue estimates for Dirac operators on compact manifolds to noncompact, even incomplete manifolds.
This includes Friedrich's estimate for manifolds with positive scalar curvature as well as the author's estimate on surfaces.
\end{abstract}
\maketitle

\section{Introduction}

This paper is concerned with spectral bounds for Dirac operators on noncompact Riemannian manifolds.
We work with generalized Dirac operators in the sense of Gromov and Lawson as explained in Section~\ref{sec:general}.
Examples are the classical Dirac operator acting on spinors, the operator $D=d+\delta$ acting on forms, and the operator $D=\sqrt{2}(\bar\partial + \bar\partial^*)$ on a K\"ahler manifold.

In Section~\ref{sec:Friedrich} we show that Friedrich's lower bound for the eigenvalues of the classical Dirac operator on a compact spin manifold of positive scalar curvature extends to a lower bound for the fundamental tone of the square of any Dirac operator even if the manifold $M$ is incomplete.
If $M$ is complete this implies a spectral gap for the Dirac operator itself.
We also discuss the equality case.
If the spectral bound is an eigenvalue of the Dirac operator, then the corresponding eigensection satisfies the Killing spinor equation, an overdetermined elliptic equation of first order.
In the case of the classical Dirac operator this implies that $M$ is compact and Einstein.
As another special case we give a simple proof of Lichnerowicz' classical lower bound for the first positive eigenvalue of the Laplace operator on a compact manifold with positive Ricci curvature.

In Section~\ref{sec:essential} we study the essential spectrum.
We show that if the curvature endomorphism of the Dirac operator is bounded from below at infinity, then a Friedrich-Lichnerowicz type estimate holds for the essential spectrum.
In particular, this yields a sufficient criterion for a Dirac operator on a complete manifold to be a Fredholm operator.
If the curvature endomorphism tends to infinity at infinity, then the spectrum of the square of the Dirac operator is discrete just as in the case of a compact manifold.

In the last section we study the classical Dirac operator on surfaces.
We show that if $M$ is a connected surface of genus $0$ with finite area and if the spin structure is bounding at infinity, then 
$$
\lambda_*(D^2) \geq \frac{4\pi}{\area(M)} .
$$
Here $M$ need not be complete and may have infinitely many ends.
If one drops the assumption on the spin structure, then the estimate fails.
The estimate is proved by reducing it to the case $M=S^2$ which was established by the author almost two decades ago.

\section{Generalized Dirac operators}
\label{sec:general}

Let $M$ denote a connected Riemannian manifold.
There will be no completeness or compactness assumption on $M$ unless we explicitly say so.
We will work with {\em generalized Dirac operators} in the sense of
Gromov and Lawson \cite{gromov-lawson83a}.
Let $E \to M$ be a Riemannian or Hermitian vector bundle over $M$ 
equipped with a metric connection $\nE$.

For $\phi,\psi\in C_c^\infty(M,E)$, the space of smooth, compactly supported sections of $E$, the $L^2$-product and norms are defined,
$$
(\phi,\psi) = \int_M \<\phi,\psi\> dV, 
\quad
\|\phi\|^2 = (\phi,\phi).
$$
Here $dV$ denotes the volume element induced by the Riemannian metric of $M$.
The Hilbert space obtained by completing $C_c^\infty(M,E)$ with respect to this scalar product is denoted by $L^2(M,E)$.
Pointwise scalar products will be denoted by $\<\cdot,\cdot\>$ and pointwise norms by $|\cdot|$.

We assume that tangent vectors act by {\em Clifford multiplication} on $E$,
i.~e.\ there is a vector bundle homomorphism
$$
TM \otimes E \to E,
\quad
X \otimes \phi \mapsto X \cdot \phi,
$$
satisfying
\begin{itemize}
\item
the Clifford relations 
$$
X \cdot Y \cdot \phi + Y  \cdot X \cdot \phi + 2 \< X,Y \> \phi = 0
$$
for all $X,Y \in T_pM$, $\phi \in E_p$, $p \in M$.
\item
skew symmetry
$$
\< X \cdot \phi, \psi \> = - \< \phi, X \cdot \psi \>
$$ 
for all $X \in T_pM$, $\phi,\psi \in E_p$, $p \in M$.
\item
the product rule
$$
\nE_X(Y \cdot \phi) = (\nabla_X Y) \cdot \phi + Y \cdot \nE_X \phi
$$
for all differentiable vector fields $X$ and $Y$ and for all differentiable
sections $\phi$ in $E$.
\end{itemize}
Here  $\nabla$ is the Levi-Civita connection of $M$.
The {\em Dirac operator} is now defined by
$$
D \phi := \sum_{i=1}^n e_i \cdot \nE_{e_i}\phi
$$
where $e_1, \ldots, e_n$ denotes any orthonormal tangent basis, $n=\dim(M)$.
This definition is independent of the choice of basis.

\begin{example}
Let $M$ carry a spin structure.
Then one can define the spinor bundle $E=\Sigma M$ which has all the required structure.
The corresponding generalized Dirac operator is the {\em classical Dirac operator}, sometimes also called the {\em Atiyah-Singer operator}.
See e.~g.\ \cite{LM} for the definitions.
\end{example}

\begin{example}\label{ex:Formen}
Let $M$ be an $n$-dimensional Riemannian manifold and $E = \bigoplus_{p=0}^n \Lambda^p T^*M$.
Using the wedge product and its adjoint one can define a Clifford multiplication such that the corresponding generalized Dirac operator is given by $D = d + \delta$.
Here $d$ is the exterior differential and $\delta$ is its adjoint.
See e.~g.\ \cite[Ch.~3]{Roe} for details.
\end{example}

\begin{example}
Let $M$ be a K\"ahler manifold of complex dimension $m$.
Then $E = \bigoplus_{q=0}^m \Lambda^{0,q}T^*M$ can be equipped with a Clifford multiplication such that the corresponding generealized Dirac operator is given by $D = \sqrt{2}(\bar{\partial} + \bar\partial^*)$, compare \cite[Ch.~3]{Roe}.
\end{example}

Any generalized Dirac operator is a formally self-adjoint elliptic differential operator of first order.
By \cite[Thm.~1.17]{gromov-lawson83a} the Dirac operator is 
essentially self-adjoint on the domain $C_c^\infty(M,E)$, smooth sections
of $E$ with compact support, in the Hilbert space $L^2(M,E)$ provided that $M$ is complete.
In this case, when we speak about the spectrum of the Dirac operator we will always
mean the spectrum of its self-adjoint closure.
The spectrum $\sigma(D)$ decomposes into
two disjoint parts, the {\em point spectrum} $\spp(D)$ consisting of all eigenvalues
(with square-integrable eigensections), and the {\em continuous spectrum}
$\scc(D)$.
By elliptic regularity theory eigensections are smooth.
If $M$ is compact, then $\scc(D)$ is empty and $\spp(D)$ is discrete
and all eigenvalues have finite multiplicity, compare Corollary~\ref{cor:sigmaDdiscrete}.

Since the operator $D^2$ is symmetric and nonnegative on the domain $C_c^\infty(M,E)$ it has a canonical self-adjoint extension even if $M$ is not complete.
This is the Friedrichs extension, the unique self-adjoint extension whose domain is contained in $H^1_D(M,E)$, the closure of $C_c^\infty(M,E)$ with respect to the norm induced by the scalar product 
$$
(\phi,\psi)_{H^1_D} = (\phi,\psi) + (D\phi,D\psi) .
$$
If $M$ is complete, then this is the only self-adjoint extension of $D^2$ and it coincides with the square of the self-adjoint extension of $D$.
If $M$ is the interior of a compact manifold $\overline M$ with boundary such that the metric extends, then taking the Friedrichs extension corresponds to imposing Dirichlet boundary conditions.
When we speak about the spectrum of $D^2$ we always mean the spectrum of the Friedrichs extension.
Again, the spectrum decomposes disjointly into the point spectrum and the continuous spectrum,
$$
\sigma(D^2) = \spp(D^2) \sqcup \scc(D^2) .
$$
Since $D^2$ is nonnegative its spectrum is contained in $[0,\infty)$.
We call $\lambda_*(D^2) := \min\sigma(D^2)$ the {\em Dirac fundamental tone}.
One has
$$
\lambda_*(D^2) 
\quad= 
\inf_{\phi\in C_c^\infty(M,E)\atop \phi\neq0}\frac{(D^2\phi,\phi)}{(\phi,\phi)}
\quad= 
\inf_{\phi\in C_c^\infty(M,E)\atop \phi\neq0}\frac{(D\phi,D\phi)}{(\phi,\phi)} 
\quad= 
\inf_{\phi\in H^1_D(M,E)\atop \phi\neq0}\frac{(D\phi,D\phi)}{(\phi,\phi)}.
$$
See \cite[Sec.~4.2]{weidmann00} for the functional analytic background.

For a differentiable section $\phi$ and a differentiable function $f:M \to \CC$ we have
\begin{equation}
D(f\phi) = \grad f \cdot \phi + f D\phi
\label{symbol}
\end{equation}
by the very definition of $D$.
The {\em curvature endomorphism} of $D$ is defined by
$$
\K := \halb \sum_{i,j=1}^n e_i \cdot e_j \cdot R^E(e_i,e_j)
$$
where $R^E$ is the curvature tensor of $\nE$.
The endomorphism field $\K$ is symmetric and we have the 
{\em Bochner-Lichnerowicz formula} \cite[Prop.~2.5]{gromov-lawson83a}
\begin{equation}
D^2 = (\nE)^*\nE + \K.
\label{bochner}
\end{equation}

If $D$ is the classical Dirac operator acting on spinors, then $\K = \frac14 \scal$.
If $D=d+\delta$, then the restriction of $\K$ to $1$-forms is Ricci curvature.

\section{A Friedrich inequality}
\label{sec:Friedrich}

Now we show that $D$ has a spectral gap provided $\K$ is 
uniformly positive, a result which is due to Friedrich in the compact case
\cite[Thm.~A]{friedrich80a}.
Since we cannot work with eigensections here, the proof needs modification.

\begin{thm}\label{friedrich}
Let $M$ be a (possibly incomplete) $n$-dimensional Riemannian manifold.
Let $D$ be a generalized Dirac operator on $M$ whose curvature endomorphism
is bounded from below by a positive constant $\kappa\in\RR$, i.~e.\
$\K \geq \kappa > 0$ in the sense of symmetric endomorphisms.
Then
$$
\lambda_*(D^2) \quad\geq\quad \frac{n\kappa}{n-1}  .
$$
\end{thm}

\begin{proof}
Let $\phi\in C_c^\infty(M,E)$.
By \gl{bochner} we have
\begin{equation}
\|D\phi\|^2 
\,\, =\,\, 
(D^2\phi,\phi)
\,\, =\,\, 
(((\nE)^*\nE + \K)\phi,\phi)
\,\,\geq\,\, 
\|\nE\phi\|^2 + \kappa \|\phi\|^2.
\label{bochnerabsch}
\end{equation}
By the Cauchy-Schwarz inequality\footnote{This trick was shown to me by O.~Hijazi who in turn attributes it to S.~Gallot, cf.\ \cite[Thm.~5.3]{Hij}.
Alternatively, one could have argued using the twistor operator.} we have
\begin{eqnarray*}
|D\phi| &=& |\sum_{i=1}^n e_i\cdot \nE_{e_i}\phi|
\quad\leq\quad
\sum_{i=1}^n |\nE_{e_i}\phi| \\
&\leq&
\sqrt{\sum_{i=1}^n 1^2}\cdot \sqrt{\sum_{i=1}^n |\nE_{e_i}\phi|^2}
\quad = \quad
\sqrt{n}\, |\nE\phi|
\end{eqnarray*}
hence
$$
|\nE\phi|^2 \quad\geq\quad \frac{1}{n} |D\phi|^2.
$$
Plugging this into \gl{bochnerabsch} we get
$$
\|D\phi\|^2 \quad\geq\quad \frac{1}{n} \|D\phi\|^2 + \kappa \|\phi\|^2
$$
thus
$$
\frac{(D\phi,D\phi)}{(\phi,\phi)} \quad\geq\quad \frac{n\kappa}{n-1} 
$$
and the theorem follows.
\end{proof}

\begin{cor}
Let $M$ be a complete $n$-dimensional Riemannian manifold.
Let $D$ be a generalized Dirac operator on $M$ whose curvature endomorphism
is bounded from below by a positive constant $\kappa\in\RR$, i.~e.\
$\K \geq \kappa > 0$ in the sense of symmetric endomorphisms.
Then
$$
\sigma(D) \subset \RR \setminus
\left(-\sqrt{\frac{n\kappa}{n-1}},\sqrt{\frac{n\kappa}{n-1}} \right) .
$$
\end{cor}

This result cannot be improved in the sense that there are examples for which $\pm \sqrt{\frac{n\kappa}{n-1}}$ lies in the spectrum of $D$.
The question arises what we can say about such examples.
In order to proceed we need the following technical lemma whose proof is given in the appendix.

\begin{lemma}\label{nablaL2}
Let $M$ be a complete $n$-dimensional Riemannian manifold.
Let $D$ be a generalized Dirac operator on $M$ whose curvature endomorphism
is bounded from below by a constant $\kappa\in\RR$ (not necessarily positive), 
i.~e.\ $\K \geq \kappa$ in the sense of symmetric endomorphisms.
Let $\phi \in C^1(M,E)$.

If $\phi$ and $D\phi$ are square-integrable, then so is $\nE\phi$ and 
$$
\|D\phi\|^2 = \|\nE\phi\|^2 + (\K\phi,\phi)
$$
holds.
\end{lemma}

Note that we do not claim that $\K\phi \in L^2(E)$ unless $\K$ is also
bounded from above.
Nevertheless, the integral $(\K\phi,\phi)$ exists.

\begin{thm}\label{Killing}
Let $M$ be a complete $n$-dimensional Riemannian manifold.
Let $D$ be a generalized Dirac operator on $M$ whose curvature endomorphism
is bounded from below by a positive constant $\kappa\in\RR$, i.~e.\
$\K \geq \kappa > 0$ in the sense of symmetric endomorphisms.
Suppose $\alpha \in \spp(D)$ where $\alpha = \sqrt{\frac{n\kappa}{n-1}}$ or $\alpha = -\sqrt{\frac{n\kappa}{n-1}}$.

Then every square-integrable eigensection of $D$ to the eigenvalue $\alpha$ satisfies the overdetermined equation
\begin{equation}\label{eq:Killing}
\nE_X \phi = -\frac{\alpha}{n}X\cdot \phi
\end{equation}
for all $X\in TM$.
\end{thm}

\begin{proof}
Let $\phi$ be an $L^2$-eigensection of $D$ to the eigenvalue $\alpha$.
By elliptic regularity theory $\phi$ is smooth.
Moreover, $\phi$ and $D\phi$ are $L^2$, so Lemma~\ref{nablaL2} applies.
Hence all estimates in the proof of Theorem~\ref{friedrich} hold for $\phi$ even though $\phi$ does not in general have compact support.
Since $\phi$ is an eigensection of $D^2$ to the eigenvalue $\alpha^2 = \frac{n\kappa}{n-1}$ all inequalities in the proof of Theorem~\ref{friedrich} must actually be equalities.
From equality in the Cauchy-Schwarz inequality 
$$
\sum_{i=1}^n |\nE_{e_i}\phi| 
\,\,\leq\,\,
\sqrt{\sum_{i=1}^n 1^2}\cdot \sqrt{\sum_{i=1}^n |\nE_{e_i}\phi|^2}
$$
we deduce
$$
|\nE_{e_1}\phi| \,\,= \cdots =\,\, |\nE_{e_n}\phi| .
$$
Equality in the triangle inequality
$$
|\sum_{i=1}^n e_i\cdot \nE_{e_i}\phi|
\quad\leq\quad
\sum_{i=1}^n |e_i\cdot\nE_{e_i}\phi|
\quad =\quad
\sum_{i=1}^n |\nE_{e_i}\phi|
$$
now implies
$$
e_1\cdot\nE_{e_1}\phi \,\,= \cdots =\,\, e_n\cdot\nE_{e_n}\phi .
$$
Hence
$$
e_1\cdot\nE_{e_1}\phi 
\,\,=\,\, 
\frac1n \sum_{i=1}^n e_i\cdot\nE_{e_i}\phi 
\,\,=\,\, 
\frac1n D\phi
\,\,=\,\, 
\frac{\alpha}{n}\phi 
$$
and therefore
$$
\nE_{e_1}\phi \,\,=\,\, -\frac{\alpha}{n}e_1\cdot\phi .
$$
Since the choice of orthonormal basis $e_1,\ldots,e_n$ is arbitrary we have 
$$
\nE_{X}\phi \,\,=\,\, -\frac{\alpha}{n}X\cdot\phi 
$$
for all tangent vectors $X$ of unit length and by linearity for all $X\in TM$.
\end{proof}

\begin{cor}
Under the assumptions of Theorem~\ref{Killing} the manifold $M$ has finite volume.
\end{cor}

\begin{proof}
From \eqref{eq:Killing} we conclude that the eigensection $\phi$ has constant length:
\begin{eqnarray*}
\partial_X |\phi|^2
&=&
\<\nE_X\phi,\phi\> + \<\phi,\nE_X\phi\> \\
&=&
-\frac\alpha n \left(\<X\cdot\phi,\phi\> + \<\phi,X\cdot\phi\>\right) \\
&=&
0.
\end{eqnarray*}
Since $\phi$ is square integrable the volume of $M$ must be finite.
\end{proof}

\begin{cor}
Let $M$ be a complete connected $n$-dimensional Riemannian spin manifold with scalar curvature $\scal \geq S$ for some positive constant $S$.
Let $D$ be the classical Dirac operator acting on spinors.
Suppose $\alpha \in \spp(D)$ where $\alpha = \frac12\sqrt{\frac{nS}{n-1}}$ or $\alpha = -\frac12\sqrt{\frac{nS}{n-1}}$.

Then $M$ is compact and Einstein with $\ric=\frac{S}{n}g$.
\end{cor}

\begin{proof}
Theorem~\ref{Killing} applies with $\K = \frac{\scal}{4}$ and $\kappa=\frac{S}{4}$.
Hence $M$ carries a nontrivial Killing spinor with real nonzero Killing number.
The corollary follows from \cite[Thm.~B]{friedrich80a} and \cite[p.~31, Thm.~9]{BFGK}.
\end{proof}

This corollary was shown by Gro\ss e with different methods in \cite[Thm.~4.0.7 (a)]{grosse}.

\begin{remark}
If $M$ is a compact Riemannian spin manifold of dimension $n\geq 3$, then there is an interesting refinement of Friedrich's inequality due to Hijazi.
Namely, one has
\begin{equation}\label{eq:Hijazi}
\lambda_*(D^2) \geq \frac{n}{4(n-1)}\lambda_*(L)
\end{equation}
where $D$ is the classical Dirac operator acting on spinors and $L=4\frac{n-1}{n-2}\Delta + \scal$ is the Yamabe operator (or conformal Laplacian) acting on functions, see \cite[Thm.~A]{Hij0}.
It was noted by Gro\ss e that this no longer holds on noncompact complete manifolds \cite[Rem.~4.2.1]{grosse}.
Hyperbolic space provides a simple counter-example.
However, \eqref{eq:Hijazi} may hold on complete manifolds of finite volume.
Partial results can be found in \cite[Thms~1.1 and 1.2]{grosse0}, see also \cite[Thms~4.0.5 and 4.2.2]{grosse}.
\end{remark}

\begin{remark}
Theorem~\ref{friedrich} can be refined as follows.
Let $D$ be a generalized Dirac operator acting on sections of a Riemannian or Hermitian vector bundle $E$ over a (possibly incomplete) Riemannian manifold $M$.
Suppose there is an orthogonal splitting $E=E_0 \oplus E_1$ with respect to which we write 
$$
D^2 = \begin{pmatrix}\Delta_{00} & \Delta_{10} \cr \Delta_{01} & \Delta_{11}\end{pmatrix}
\quad\mbox{ and }\quad
\K = \begin{pmatrix}\K_{00} & \K_{10} \cr \K_{01} & \K_{11}\end{pmatrix} .
$$
Then the proof of Theorem~\ref{friedrich} with $\phi\in C_c^\infty(M,E_0)$ shows that if $\K_{00} \geq  \kappa > 0$, then 
$$
\lambda_*(\Delta_{00}) \,\,\geq\,\, \frac{n\kappa}{n-1}
$$
where $n=\dim(M)$.
\end{remark}

\begin{example}\label{ex:Delta1}
Let $E=\bigoplus_{p=0}^n \Lambda^p T^*M$ as in Example~\ref{ex:Formen}, $E_0=\Lambda^1 T^*M = T^*M$, $E_1 = \bigoplus_{p\neq 1} \Lambda^p T^*M$, and $D=d+\delta$.
Then $\Delta_{00}=d\delta+\delta d$ is the Hodge-Laplacian acting on $1$-forms and $\K_{00}= \Ric$.
Hence if $\Ric \geq \kappa$ for a positive constant $\kappa$, then 
$$
\lambda_*(d\delta+\delta d \mbox{ on $1$-forms}) \,\,\geq\,\, \frac{n\kappa}{n-1} .
$$
\end{example}

In the case that $M$ is compact this is a classical theorem by Lichnerowicz
\cite[p.~145]{Lic}.
For functions we conclude

\begin{cor}[Lichnerowicz]
Let $M$ be a compact $n$-dimensional Riemannian manifold.
Suppose $\ric \geq \kappa$ where $\kappa$ is a positive constant.
Let $\Delta$ be the Laplace-Beltrami operator acting on functions and $\lambda_1(\Delta)$ its first nonzero eigenvalue.
Then
$$
\lambda_1(\Delta) \,\,\geq\,\, \frac{n\kappa}{n-1}.
$$
\end{cor}

\begin{proof}
Let $\phi\in C^\infty(M)$ be an eigenfunction of $\Delta$ to the eigenvalue $\lambda_1(\Delta)$.
Since $\lambda_1(\Delta)>0$ the differential $d\phi \in C^\infty(M, \Lambda^1T^*M)$ is not identically zero.
Denote the Hodge-Laplacian on $1$-forms by $\Delta_1$.
Thus be Example~\ref{ex:Delta1}
$$
\frac{n\kappa}{n-1}
\,\, \leq \,\,
\lambda_*(\Delta_1)
\,\, \leq \,\,
\frac{(\Delta_1 d\phi,d\phi)}{(d\phi,d\phi)}
\,\, = \,\,
\frac{(\Delta^2 \phi,\phi)}{(\Delta\phi,\phi)}
\,\, = \,\,
\lambda_1(\Delta) .
$$
\end{proof}

Note that if $M$ is complete and satisfies $\ric \geq \kappa > 0$, then $M$ is compact by the Bonnet-Myers theorem.
A Lichnerowicz type estimate for functions on incomplete manifolds is not to be expected.

\begin{example}
The $2$-sphere $S^2$ with its canonical metric $g$ has constant Gauss curvature $1$, hence $\ric = g$.
Indeed, $\lambda_1(\Delta) = 2 = \frac{2\cdot 1}{2-1}$ on $(S^2,g)$.
Removing north and south pole we obtain an incomplete surface $M_1$ diffeomorphic to $(-\frac{\pi}{2},\frac{\pi}{2}) \times \RR/2\pi\ZZ$ with Riemannian metric $g=d\theta^2 + \cos^2(\theta)d\phi^2$ where $\theta\in (-\frac{\pi}{2},\frac{\pi}{2})$ and $\phi\in\RR/2\pi\ZZ$.
For $k\in\NN$ let $M_k = (-\frac{\pi}{2},\frac{\pi}{2}) \times \RR/2k\pi\ZZ$ be the $k$-fold covering of $M_1$ with the pull-back metric.
Again, $M_k$ is an incomplete surface with constant Gauss curvature $1$.

\begin{center}
\includegraphics[width=35mm]{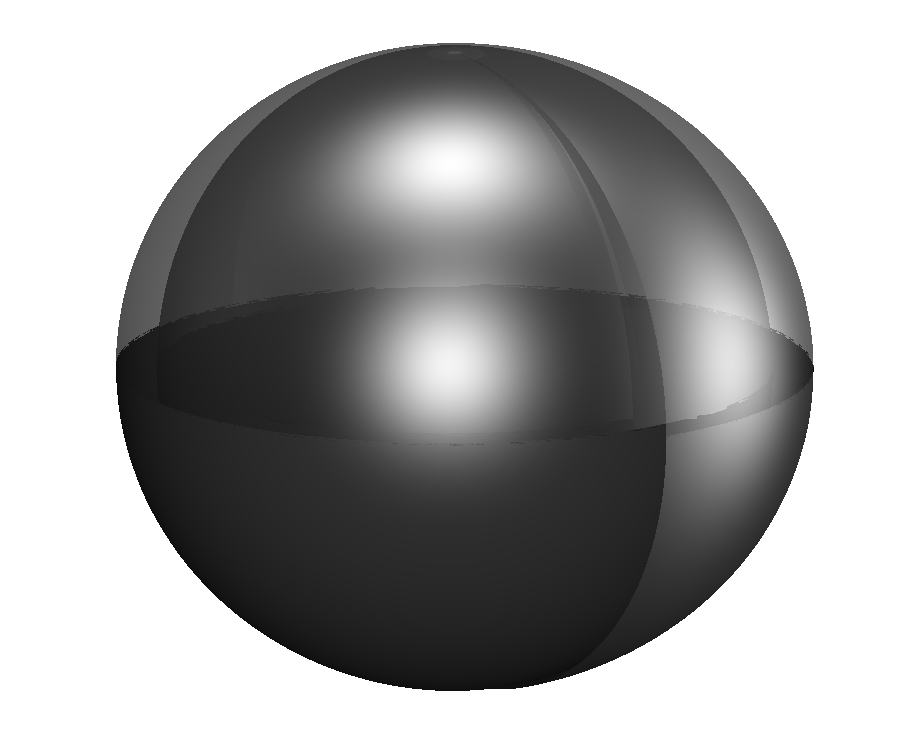}
$\begin{array}{c}
\hspace{3mm}\xrightarrow{\hspace{2mm}2:1\hspace{2mm}}\hspace{3mm}\\
\vspace{2cm}\\
\end{array}$
\includegraphics[width=30mm]{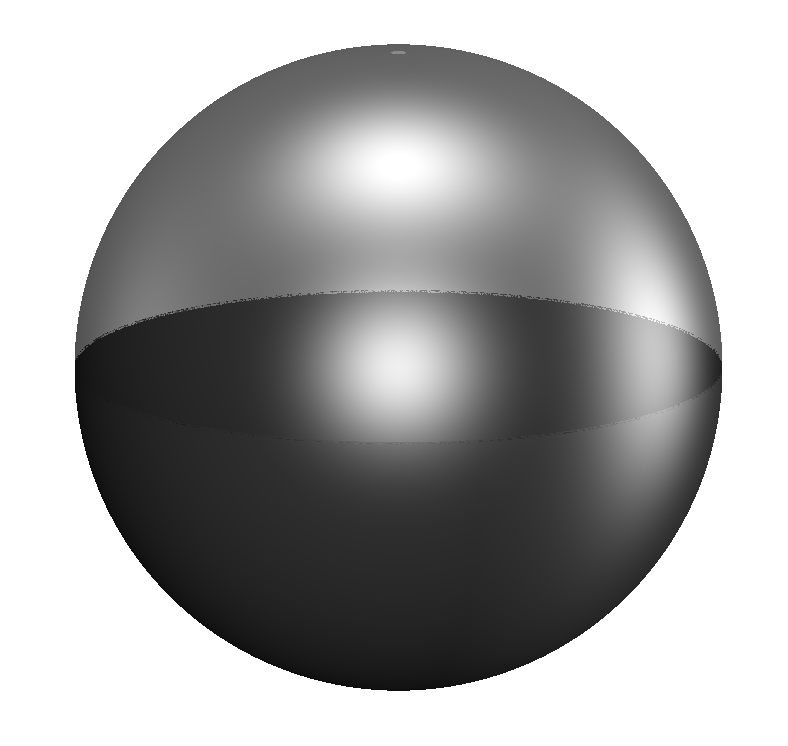}

\vspace{-15mm}
\hspace{4mm}
$M_2$
\hspace{48mm}
$M_1$

{\em Fig.~1}
\end{center}

Now $f_k(\theta,\phi) = \cos(\theta)\cos(\phi/k)$ is a well-defined smooth function on $M_k$.
Since $f_k$ is bounded and $M_k$ has finite area $f_k$ is square integrable.
It is not hard to check that $f_k \in H^1_{d+\delta}(M_k)$ and can hence be used as a test function for $\Delta$.
One easily computes
$$
(f_k,f_k)_{L^2(M_k)} 
=
\int_{-\pi/2}^{\pi/2}\int_{0}^{2k\pi}\cos^2(\theta)\cos^2(\phi/k)\cos(\theta)d\phi d\theta
=
\frac{4k\pi}{3},
$$
$$
\Delta f_k 
=
\left(\tan(\theta)\frac{\partial}{\partial\theta} - \frac{\partial^2}{\partial\theta^2} - \frac{1}{\cos^2(\theta)}\frac{\partial^2}{\partial\phi^2}\right)f_k
=
2f_k -(1-k^{-2})\frac{\cos(\phi/k)}{\cos(\theta)}
$$
and hence
$$
\frac{(\Delta f_k,f_k)_{L^2(M_k)}}{(f_k,f_k)_{L^2(M_k)}}
=
2 - \frac32 (1-k^{-2}) .
$$
For $k\geq 2$ this Rayleigh quotient is smaller than $2$ despite the fact that $f_k$ is $L^2$-perpendicular to the constant functions, i.e., to the kernel of $\Delta$.
In this sense, the Lichnerowicz estimate is violated on $M_k$ for $k\geq 2$.
\end{example}

\section{The essential spectrum}
\label{sec:essential}

The subset of $\spp(D^2)$ consisting of eigenvalues of finite multiplicity which are isolated in $\sigma(D^2)$ is called the {\em discrete spectrum} of $D^2$ and is denoted by $\sd(D^2)$.
Its complement in $\sigma(D^2)$ is called the {\em essential spectrum} and is denoted by $\se(D^2)$.
The essential spectrum is unaffected by changes on the manifold and the operator in a compact region.
More precisely, if $M$ and $\tilde M$ are Riemannian manifolds equipped with generalized Dirac operators $D$ and $\tilde D$ and if there are compact subsets $K\subset M$ and $\tilde K \subset \tilde M$ such that $M\setminus K=\tilde M \setminus \tilde K$ and $D=\tilde D$ outside $K$ and $\tilde K$ resp., then $\se(D^2) = \se(\tilde D^2)$.
In the literature this is known as the {\em decomposition principle}, see e.~g.\ \cite[Prop.~1]{Baer2000} for a proof.

We say that the curvature endomorphism $\K$ tends to $\infty$ at $\infty$, in symbols $\lim_{x\to\infty}\K(x)=\infty$, if for each $\kappa\in\RR$ there exists a compact subset $K\subset M$ such that $\K(x)\geq\kappa$ for all $x\in M\setminus K$.
Given $\kappa\in\RR$ we write $\liminf_{x\to\infty}\K(x) \geq \kappa$ if for each $\epsilon>0$ there exists a compact subset $K\subset M$ such that $\K(x)\geq\kappa-\epsilon$ for all $x\in M\setminus K$.

\begin{thm}\label{friedrich-ess}
Let $M$ be a (possibly incomplete) $n$-dimensional Riemannian manifold.
Let $D$ be a generalized Dirac operator on $M$ whose curvature endomorphism
satisfies $\liminf_{x\to\infty}\K(x) \geq \kappa$ for some positive constant $\kappa$.
Then
$$
\se(D^2) \subset \left[\frac{n\kappa}{n-1},\infty\right)  .
$$
\end{thm}

\begin{proof}
Let $\epsilon >0$.
We choose a compact subset $K\subset M$ such that $\K(x)\geq\kappa-\epsilon$ for all $x\in M\setminus K$.
We put $\tilde M:= M \setminus K$ and we let $\tilde D$ be the restriction of $D$ to $\tilde M$.
By the decomposition principle and by Theorem~\ref{friedrich} we get
$$
\se(D^2) = \se(\tilde D^2) \subset \left[\frac{n(\kappa-\epsilon)}{n-1},\infty\right) .
$$
Taking the limit $\epsilon \searrow 0$ concludes the proof.
\end{proof}

\begin{cor}
Let $M$ be a complete $n$-dimensional Riemannian manifold.
Let $D$ be a generalized Dirac operator on $M$ whose curvature endomorphism
satisfies $\liminf_{x\to\infty}\K(x) \geq \kappa$ for some positive constant $\kappa$.
Then
$$
\se(D) \subset \RR \setminus \left(-\sqrt{\frac{n\kappa}{n-1}},\sqrt{\frac{n\kappa}{n-1}}\right)  .
$$
In particular, $D$ is Fredholm.
\end{cor}

An entirely different proof based on elliptic estimates of the finite dimensionality of the kernel of $D$ can be found in \cite[Thm.~3.2]{gromov-lawson83a}.
This Fredholm property is crucial for index theory of Dirac operators.

\begin{cor}
Let $M$ be a (possibly incomplete) Riemannian manifold.
Let $D$ be a generalized Dirac operator on $M$ whose curvature endomorphism
satisfies $\lim_{x\to\infty}\K(x)=\infty$.

Then the spectrum of $D^2$ is discrete,
$$
\sigma(D^2) = \sd(D^2).
$$
\end{cor}

\begin{proof}
Since $\liminf_{x\to\infty}\K(x) \geq \kappa$ for any $\kappa$, Theorem~\ref{friedrich-ess} shows that $\se(D^2)=\emptyset$.
\end{proof}

\begin{cor}\label{cor:sigmaDdiscrete}
Let $M$ be a complete Riemannian manifold.
Let $D$ be a generalized Dirac operator on $M$ whose curvature endomorphism
satisfies $\lim_{x\to\infty}\K(x)=\infty$.

Then the spectrum of $D$ is discrete,
$$
\sigma(D) = \sd(D).
$$
\end{cor}

This corollary applies in particular when $M$ is compact.
In this case it is well-known of course.
In \cite[Thm.~1]{Baer2000} it is shown that the essential spectrum of the classical Dirac operator on a complete hyperbolic spin manifold of finite volume is either empty or the whole real line.
It is a topological property of the spin structure which determines the essential spectrum.
This discreteness criterion for the Dirac spectrum on hyperbolic manifolds was generalized by S.~Moroianu to manifolds with certain cusp-like ends, see \cite[Thm.~2]{moroianu}. 
There is also an interesting Weyl asymptotics for the growth of the eigenvalues in this case \cite[Thm.~3]{moroianu}.

\section{Surfaces}
\label{sec:surfaces}

In this section we consider connected $2$-dimensional manifolds $M$.
We say that $M$ has genus $0$ if $M$ is diffeomorphic to an open subset of $S^2$.
A spin structure on $M$ is said to be {\em bounding at infinity} if $M$ can be embedded into $S^2$ in such a way that the spin structure extends to the unique spin structure of $S^2$.

A circle $S^1$ has two inequivalent spin structures.
Only one of them extends to the unit disc in $\RR^2$ which is bounded by $S^1$.
We call this the {\em bounding spin structure} of $S^1$.
Sometimes this is called the nontrivial spin structure of $S^1$.

If $K\subset M$ is a compact subset with smooth boundary, then the boundary of $K$ is a disjoint union of circles.
If there is a compact set $K\subset M$ such that $M\setminus K$ is diffeomorphic to the disjoint union of finitely many copies of $\RR\times S^1$, then we say that $M$ has finitely many ends.
In this case a spin structure on $M$ is bounding at infinity if and only if the induced spin structure on the circles $\{t_0\}\times S^1$ is bounding for each end.

\begin{center}
\includegraphics[width=50mm]{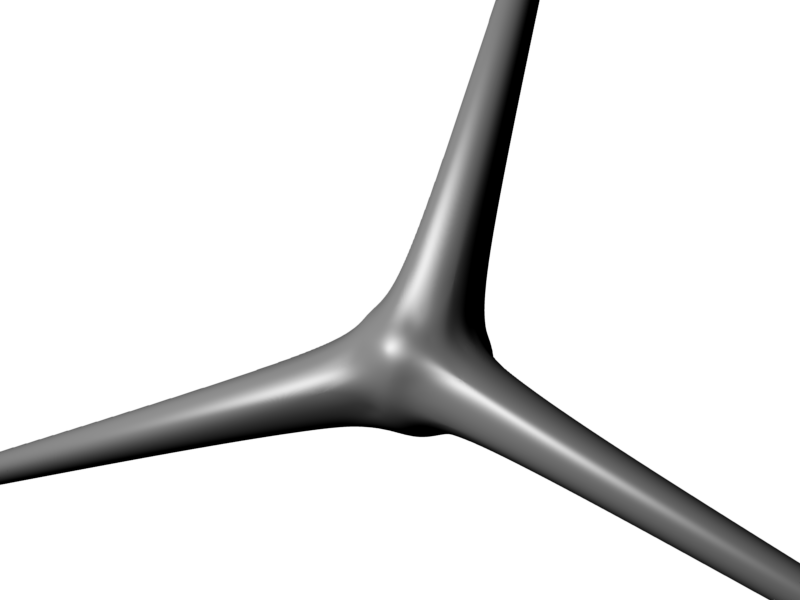}

A surface of genus $0$ with $3$ ends

{\em Fig.~2}
\end{center}

If $M$ is of genus $0$ and has one end, i.~e.\ $M$ is diffeomorphic to $\RR^2$, then its unique spin structure is bounding at infinity.
If $M$ is of genus $0$ and has two ends, i.~e.\ $M$ is diffeomorphic to $S^1\times \RR$, then $M$ has two spin structures, one of which is bounding at infinity while the other is not.

\begin{thm}\label{baer}
Let $M$ be a connected $2$-dimensional Riemannian manifold of genus $0$ and with finite area.
Let $M$ be equipped with a spin structure which is bounding at infinity.
Let $D$ be the classical Dirac operator acting on spinors on $M$.
Then
$$
\lambda_*(D^2) \geq \frac{4\pi}{\area(M)}
$$
\end{thm}
Note that $M$ need not be complete and may have infinitely many ends.
In case $M$ is diffeomorphic to $S^2$ this inequality was conjectured in \cite{lott0} and first shown by the author in \cite[Satz~1.10 on p.~93]{Diss}, see also \cite[Thm.~2]{Baer92}.

\begin{proof}
Let $M$ be as in the statement of the theorem and let $\phi\in C^\infty_c(M,\Sigma M)$.
We embed $M$ into $S^2$ such that the spin structure extends.
We choose a Riemannian metric $g$ on $S^2$ such that the embedding is isometric on the compact set $K:=\supp(\phi)$ and such that $\area(S^2,g)\leq\area(M)$.
Since $\phi$ vanishes outside $K$ we can regard $\phi$ as a smooth spinor on $S^2$.
Therefore by the inequality in the compact case
$$
\frac{(D\phi,D\phi)}{(\phi,\phi)} \geq \lambda_*(D^2,S^2,g) \geq \frac{4\pi}{\area(S^2,g)} \geq \frac{4\pi}{\area(M)}
$$
and the theorem follows.
\end{proof}

Since the spin structure on a surface of genus $0$ with exactly one end is automatically bounding at infinity we get

\begin{cor}
Let $M$ be a connected $2$-dimensional Riemannian manifold of genus $0$, with exactly one end, and with finite area.
Let $D$ be the classical Dirac operator acting on spinors on $M$.
Then
$$
\lambda_*(D^2) \geq \frac{4\pi}{\area(M)} .
$$
\end{cor}

This has first been shown by N.~Gro\ss e using different methods in \cite[Thm.~3.1.7]{grosse}.

\begin{remark}
The assumption on the spin structure in Theorem~\ref{baer} cannot be dispensed with.
For example, let $M = S^1 \times (0,L)$ with the product metric $d\theta^2 + dt^2$ such that $S^1$ has length $2\pi$.
The two spin structures on $S^1$ induce two spin structures on $M$.
We pick the spin structure coming from the non-bounding spin structure on $S^1$.
With respect to this spin structure $M$ carries a non-trivial parallel spinor $\psi$.
Put
\begin{equation}\label{eq:defphi}
\phi(\theta,t) := \sin(\pi t/L)\psi(\theta,t) .
\end{equation}
Then $\phi$ is smooth and satisfies Dirichlet boundary conditions on $\overline M = S^1 \times [0,L]$.
Hence $\phi$ is in the domain of the Friedrichs extension of $D^2$.
From
$$
D^2\phi = \frac{\pi^2}{L^2}\phi
$$
we see that
$$
\lambda_*(D^2) 
\leq 
\frac{(D^2\phi,\phi)}{(\phi,\phi)}
=
\frac{\pi^2}{L^2}
<
\frac{4\pi}{\area(M)}
$$
if $L > \frac{\pi^2}{2}$.

This example is an open cylinder and hence incomplete.
It is easy to modify this example such that it becomes complete.
To do this we attach cusps of finite area $A$ to both boundary components of $S^1\times [0,L]$.
Then $\area(M) = 2\pi L + 2A$.

\begin{center}
\includegraphics[height=20mm]{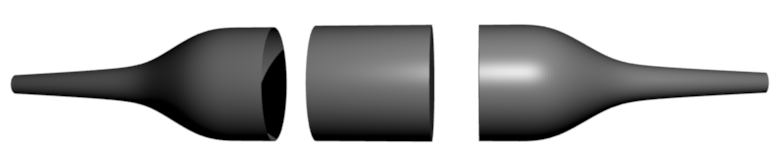}

\vspace{-5mm}
cusp \hspace{12mm} $S^1\times L$ \hspace{12mm} cusp

\includegraphics[height=23mm]{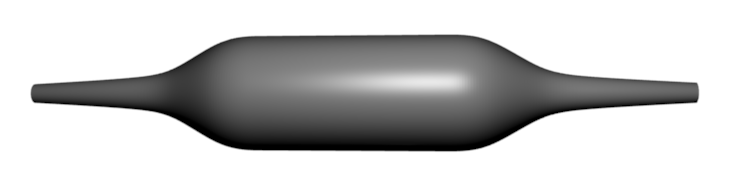}

\vspace{-8mm}
$M$

{\em Fig.~3}
\end{center}

Choose $\phi$ as in \eqref{eq:defphi} on $S^1\times [0,L]$ and extend it by zero to all of $M$.
Then $\phi$ does not lie in the domain of the Friedrichs extension of $D^2$ but it does lie in the domain $H^1_D(M,\Sigma M)$ of the unique self-adjoint extension of $D$ and 
$$
\lambda_*(D^2) 
\leq 
\frac{(D\phi,D\phi)}{(\phi,\phi)}
=
\frac{\pi^2}{L^2}
<
\frac{4\pi}{2\pi L + 2A}
=
\frac{4\pi}{\area(M)}
$$
if $L\gg 0$.
\end{remark}

\begin{remark}
In the terminology of \cite{Baer2000} a spin structure bounding at infinity is a spin structure which is nontrivial along all cusps.
Theorem~1 in \cite{Baer2000} says that the spectrum of the classical Dirac operator on a complete hyperbolic spin manifold with finite volume is discrete if and only if the spin structure is nontrivial along all cusps.

In \cite{mo-mo} one finds criteria for the discrete spectrum of the classical Dirac operator to be empty.
\end{remark}

\section{appendix}

It remains to provide a proof of the technial Lemma~\ref{nablaL2}.

\begin{proof}[Proof of Lemma~\ref{nablaL2}]
Since $\phi$ and $D\phi$ are $L^2$ we have that $\phi$ lies in the
domain of the maximal extension of $D$.
Since $M$ is complete $D$ is essentially self-adjoint, hence the maximal and the minimal
extensions coincide, so there exist $\phi_k \in C_c^\infty(M,E)$
such that $\phi_k \to \phi$ and $D\phi_k \to D\phi$ in $L^2(M,E)$
for $k\to \infty$.
By \gl{bochnerabsch} $\|\nE\phi_k\|^2 \leq \|D\phi_k\|^2 - \kappa
\|\phi_k\|^2 \to \|D\phi\|^2 - \kappa \|\phi\|^2$.
In particular, there is a constant $C>0$ such that $\|\nE\phi_k\| \leq C$
for all $k$.
For each $\psi\in C_c^\infty(M,TM\otimes E)$ we see that
$$
(\psi,\nE\phi) = ((\nE)^*\psi, \phi) = \lim_{k\to \infty}((\nE)^*\psi, \phi_k)
= \lim_{k\to \infty}(\psi,\nE\phi_k) ,
$$
hence
$$
|(\psi,\nE\phi)| \leq C\, \|\psi\| .
$$
Therefore the linear functional $C_c^\infty(M,TM\otimes E) \to \CC,\, 
\psi \mapsto (\psi,\nE\phi),$ extends to a bounded functional 
$L^2(M,TM\otimes E) \to \CC$.
By the Riesz representation theorem there exists $\chi\in L^2(M,TM\otimes E)$
such that $(\psi,\nE\phi) = (\psi,\chi)$ for all $\psi\in 
C_c^\infty(M,TM\otimes E)$.
Thus $\nE\phi = \chi$ is square-integrable.

For $\rho > 0$ choose a smooth cut-off function $f_\rho : M \to \RR$
with compact support satisfying
\begin{itemize}
\item
$0 \leq f_\rho \leq 1$ on $M$,
\item
$f_\rho \equiv 1$ on the ball of radius $\rho$ about some fixed point,
\item
$|\grad f_\rho| \leq 1/\rho$ on $M$.
\end{itemize}
Then we have
\begin{eqnarray*}
\|\nE(f_\rho \phi) - \nE \phi\| 
&=& 
\| df_\rho \otimes \phi + f_\rho\nE\phi - \nE \phi\| \\
&\leq& 
\| df_\rho\|_{L^\infty} \|\phi\| +  \| f_\rho\nE\phi - \nE \phi\| \\
&\leq& 
\|\phi\|/\rho + \|(f_\rho-1)\nE\phi\| 
\quad \xrightarrow{\rho\to\infty} \quad 0.
\end{eqnarray*}
Using \gl{symbol} we see similarly that $D(f_\rho \phi) \to D\phi$ in
$L^2(M,E)$ as $\rho \to \infty$.
Moreover, splitting $(\K(f_\rho\phi),f_\rho\phi) = \int_M f_\rho^2
\<\K\phi,\phi\> \dV$ into integrals over the subsets of $M$ on which
$\<\K\phi,\phi\>$ is positive and negative respectively and applying
the theorem of monotone convergence we see that $(\K(f_\rho\phi),f_\rho\phi)
\to (\K\phi,\phi)$.
Taking the limit $\rho\to\infty$ in
$$
\|D(f_\rho\phi)\|^2 = \|\nE(f_\rho\phi)\|^2 + (\K(f_\rho\phi),f_\rho\phi)
$$
yields
$$
\|D\phi\|^2 = \|\nE\phi\|^2 + (\K\phi,\phi).
$$
\end{proof}


\begin{thebibliography}{1}

\bibitem{Diss}
C.~B\"ar, \emph{Das Spektrum von Dirac-Operatoren}, Dissertation, Universit\"at Bonn 1990, Bonner Mathematische Schriften~\textbf{217} (1991)

\bibitem{Baer92}
C.~B\"ar, \emph{Lower eigenvalue estimate for Dirac operators}, Math.\ Ann.~\textbf{293} (1992), 39--46

\bibitem{Baer2000}
C.~B\"ar, \emph{The Dirac operator on hyperbolic manifolds of finite volume},
J.~Diff.~Geom.~\textbf{54} (2000), 439--488

\bibitem{BFGK}
H.~Baum, T.~Friedrich, R.~Grunewald, I.~Kath,
  \emph{Twistor and Killing spinors on Riemannian manifolds}, Teubner Verlag, Stuttgart-Leipzig, 1991

\bibitem{friedrich80a}
T.~Friedrich, \emph{Der erste {E}igenwert des {D}irac-{O}perators einer
  kompakten {R}iemannschen {M}annigfaltigkeit nicht-negativer {K}r{\"u}mmung},
  Math.\ Nachr.~\textbf{97} (1980), 117--146

\bibitem{gromov-lawson83a}
M.~Gromov and H.~B. Lawson, \emph{Positive scalar curvature and the {D}irac
  operator on complete {R}iemannian manifolds}, Publ.\ Math.\ Inst.\ Hautes Etud.\
  Sci.~\textbf{58} (1983), 295--408

\bibitem{grosse0}
N.~Gro\ss e, \emph{The Hijazi inequality on conformally parabolic manifolds},
  \texttt{arXiv:0804.3878v2}

\bibitem{grosse}
N.~Gro\ss e, \emph{On a spin conformal invariant on open manifolds},
  Dissertation, Universit\"at Leipzig 2008

\bibitem{Hij0}
O.~Hijazi, \emph{A Conformal Lower Bound for the Smallest Eigenvalue of the Dirac Operator and Killing Spinors},
Commun.\ Math.\ Phys.~\textbf{104} (1986), 151--162

\bibitem{Hij}
O.~Hijazi, \emph{Spectral properties of the Dirac operator and geometrical structures}, Proceedings of the Summer School on Geometric Methods in Quantum Field Theory, Villa de Leyva, Colombia, July 12--30, (1999), World Scientific, 2001

\bibitem{LM}
H.~B.~Lawson, M.-L.~Michelsohn,
\emph{Spin geometry}, Princeton University Press, 1989

\bibitem{Lic}
A.~Lichnerowicz,
\emph{Geometry of groups of transformations}, Noordhoff International Publishing, Leyden, 1977

\bibitem{lott0}
J.~Lott,
\emph{Eigenvalue bounds for the Dirac operator},
Pacific J.\ Math.~\textbf{125} (1986), 117--126

\bibitem{mo-mo}
A.~Moroianu, S.~Moroianu,
\emph{The Dirac spectrum on manifolds with gradient conformal vector fields},
J.\ Funct.\ Anal.~\textbf{253} (2007), 207--219

\bibitem{moroianu}
S.~Moroianu,
\emph{Weyl laws on open manifolds},
Math.\ Ann.~\textbf{340} (2008), 1--21

\bibitem{Roe}
J.~Roe,
\emph{Elliptic operators, topology and asymptotic methods}, second edition, 
Addison Wesley Longman, Harlow, 1998

\bibitem{weidmann00}
J.~Weidmann, \emph{Lineare Operatoren in Hilbertr\"aumen, Teil I Grundlagen},
  Teubner Verlag, Stuttgart-Leipzig-Wiesbaden, 2000

\end{thebibliography}
\end{document}